\documentclass[12pt,twoside]{amsart}

\usepackage{mathtools}
\mathtoolsset{showonlyrefs}

\usepackage{amsmath}
\usepackage{amssymb}
\usepackage{amsfonts}
\usepackage{amsthm}
\usepackage{bbm}
\usepackage[top=1in, bottom=1in, left=1in, right=1in]{geometry}
\theoremstyle{plain}
\newtheorem{theorem}{Theorem}[section]
\newtheorem{lemma}[theorem]{Lemma}
\newtheorem{proposition}[theorem]{Proposition}

\newtheorem{conjecture}[theorem]{Conjecture}
\newtheorem*{theorem*}{Theorem}
\newtheorem*{lemma*}{Lemma}
\newtheorem*{proposition*}{Proposition}
\newtheorem*{corollary*}{Corollary}
\newtheorem*{definition*}{Definition}
\newtheorem*{conjecture*}{Conjecture}

\numberwithin{equation}{section}

\usepackage{hyperref} 
\hypersetup{
    colorlinks=true,       
    linkcolor=blue,          
    citecolor=magenta,        
    filecolor=magenta,      
    urlcolor=cyan           
}

\title{On the Fejes T\'{o}th Problem about the sum of angles between lines}

\author{Dmitriy Bilyk}
\address{School of Mathematics, University of Minnesota, Minneapolis, MN 55408, USA.}
\email{dbilyk@math.umn.edu, matzk053@umn.edu}

\author{Ryan W Matzke}

\begin{document}

\maketitle

\begin{abstract}
In 1959 Fejes T\'oth posed a conjecture that the sum of pairwise non-obtuse angles between   $N$ unit vectors in $\mathbb S^d$  is maximized by periodically repeated elements of the standard orthonormal basis. We obtain new improved upper bounds for this sum, as well as for the corresponding energy integral. We also provide several new approaches to the only settled case of the conjecture: $d=1$. 
\end{abstract}

\section{Introduction}

 Fejes  T\'{o}th  has formulated a  variety of problems and conjectures about point distributions on the sphere. In particular, in 1959  he posed the following question \cite{A}:  
what is the maximal  value of the  sum of pairwise acute angles defined by $N$ vectors in the sphere $\mathbb S^d$?  More precisely, determine which $N$-element point sets $ Z = \{ z_1, \dots, z_N \}  \subset\mathbb{S}^d$ maximize the discrete energy 
\begin{equation}\label{e.DE}
E(Z) = \frac{1}{N^2} \sum_{i,j = 1}^{N} \arccos  | z_i \cdot z_j | =  \frac{1}{N^2} \sum_{i,j = 1}^{N}  \min \big\{ \arccos(  z_i \cdot z_j ), \pi - \arccos(  z_i \cdot z_j )   \big\}.
\end{equation}
He conjectured that this sum is maximized by the periodically repeated copies of the standard orthonormal basis:
\begin{conjecture}[Fejes  T\'{o}th, 1959 \cite{A}]\label{c1}  Let $d\ge 1$ and $N = m(d+1) + k$ with $ m \in \mathbb N_0$ and $0 \leq k \leq d$. Then the discrete energy \eqref{e.DE} on the $d$-dimensional sphere $\mathbb S^d$ 
is maximized  by the point  set $Z= \{z_1, \dots, z_N \} \subset \mathbb S^d$ with $z_{p(d+1)+ i } = e_i$, where $\{e_i \}_{i=1}^{d+1}$ is the standard orthonormal basis of $\mathbb R^{d+1} $, i.e. $Z$  consists  of $m+1$ copies of $k$ elements of the orthonormal basis of $\mathbb{R}^{d+1}$ and $m$ copies of the remaining $d+1-k$ basis elements. In this case, 
\begin{equation}\label{e.c1}
E (Z) = \frac{\pi}{2N^2} \left(k(k-1)(m+1)^2 +  2km (d+1-k)(m+1) + (d-k)(d+1-k)m^2  \right).
\end{equation}
In particular, if $N = m (d+1)$, the sum is maximized by $m$ copies of the orthonormal basis: 
 \begin{equation}\label{e.c1a}
 \max_{ \substack{Z\subset \mathbb S^d  \\ \# Z = N} } E (Z) =  \frac{\pi }{2  } \cdot \frac{ d}{d+1 } .
 \end{equation}
\end{conjecture}
This conjecture has been independently  stated in  \cite{H}  for all $d\ge 1$ (Fejes T\'oth originally formulated it just for $\mathbb S^2$).  \\

We also formulate  a continuous version of the conjecture. Let $\mathcal{B}( \mathbb{S}^d)$ denote the set of Borel probability measures on the sphere, i.e. $\mu \ge 0$ and $\mu (\mathbb S^d) =1$, and let $F:[-1,1] \rightarrow \mathbb R$ be a (bounded or positive) measurable function. For $ \mu \in \mathcal{B}( \mathbb{S}^d)$ we define the energy integral  of $\mu$ with respect to the potential  $F$ as 
\begin{equation}\label{e.int}
 I_F( \mu) = \int_{ \mathbb{S}^d} \int_{ \mathbb{S}^d} F( x \cdot y) d \mu(x) d \mu(y).
\end{equation}
In this notation, the discrete energy of a set $Z = \{ z_1, \dots, z_N \}$  is 
\begin{equation}\label{e.en}
E_F(Z) =\frac{1}{N^2} \sum_{i, j = 1}^{N} F( z_i \cdot z_j) =  I_F \left( \frac1{N} \sum_{k=1}^{N} {\delta_{z_k} } \right) . 
\end{equation}
In our case, $F(t) = \arccos |t|$ and we  suppress the  dependence on $F$ in the subscript, as we did in \eqref{e.DE}, i.e. 
\begin{equation}\label{e.int1}
I (\mu ) = \int_{ \mathbb{S}^d} \int_{ \mathbb{S}^d} \arccos | x \cdot y| \, d \mu(x) d \mu(y)
\end{equation}
In accordance to Conjecture \ref{c1} it is natural to guess that $I (\mu)$ should be maximized by the measure whose mass is equally distributed between the elements  of the orthonormal basis:
\begin{equation}\label{e.ONB}
\nu_{ONB} = \frac{1}{d+1}  \sum_{i=1}^{d+1} e_i.
\end{equation}
\begin{conjecture}\label{c2}
The energy integral $I(\mu)$ is maximized by $\nu_{ONB}$:
\begin{equation}\label{e.c2}
\max_{\mu \in \mathcal B (\mathbb S^d ) } I(\mu) =  I \big( \nu_{ONB} \big) = \frac{\pi }{2  } \cdot \frac{ d}{d+1 }. 
\end{equation}
\end{conjecture}
Since the maximal value in \eqref{e.c1a} is independent of $N$, a simple argument based on the weak$^*$-density of the span of Dirac masses in $\mathcal B (\mathbb S^d)$ shows that   the case $N= m (d+1)$ of Conjecture \ref{c1} implies Conjecture \ref{c2}, and the converse implication is obvious.  

The case $d=1$ of Conjecture \ref{c1} has been settled in \cite{B,H}. We shall return to this case in \S \ref{s.d1} and shall give several alternative proofs. 

In dimension $d=2$, Fejes T\'oth confirmed Conjecture \ref{c1} for $N \le 6$ and established an asymptotic upper bound  $E(Z) \le \frac{2\pi}{5}$ for large $N$. In \cite{B} Fodor, V\'igh, and Zarn\'ocz proved that for any point distribution $Z \subset \mathbb S^2$ 
\begin{equation}\label{e.fvz}
 E  (Z) \le \frac{3\pi}8 \,\,\, \textup{ when $N$ is even, }
 \end{equation} with a small correction for $N$ odd.  
 
In the present paper, we prove a new upper bound for $I(\mu)$ in all dimensions $d\ge 2$, which  is stronger than \eqref{e.fvz} when restricted to $\mathbb S^2$. 
\begin{theorem}\label{t.main}
In all dimensions $d\ge 2$
\begin{equation}\label{e.main}
 \max_{ \mu \in \mathcal{B} (\mathbb S^d)  } I ( \mu) \leq  \frac{\pi}{2} - \frac{69}{50 (d+1)}.
 \end{equation}
 In particular, for $d=2$,
 \begin{equation}\label{e.maind2}
 \max_{ \mu \in \mathcal{B} (\mathbb S^2)  } I ( \mu) \leq  \frac{\pi}{2} - \frac{69}{150 } = 1.110796... \,  <  \frac{3\pi}{8} = 1.178097...\,  .
 \end{equation}
\end{theorem}
We recall that, according to \eqref{e.c1a} and   \eqref{e.c2},  the conjectured maximal value in dimension    $d=2$ is $\frac{\pi}3 = 1.047198...\,$.  The proof of Theorem \ref{t.main} is based on bounding the potential $F(t) = \arccos |t|$ above by a quadratic polynomial and employing the estimates for the so-called ``frame potential".  It is presented in \S \ref{s.proof}. 

It is worth noting  that the energy optimization problem considered here has a distinctly different flavor compared to many standard problems of this kind, see e.g. \cite{BHSbook}.   In the vast majority of energy minimization problems, the potential $F(t)$ is maximized at $t=1$ and minimized at $t=-1$. 
Such is the case, for example, 
for the classical Riesz potential $F(x\cdot y) = \| x - y \|^{-s}$.  If one loosely interprets points as ``electrons", and measures as charge distributions, this results in the strongest repulsion when the electrons are close to each other, and the weakest when they are far away. In this situation one often expects $I_F (\mu)$ to be minimized by $\sigma$, the normalized Lebesgue surface measure on $\mathbb S^d$ -- in other words, minimizing energy induces uniform distribution. 

In our setting,  however, the smallest value of $-F(t)$ is at $t=0$, hence the weakest   repulsion occurs when two electrons are positioned orthogonally to each other (notice that we are maximizing, not minimizing the energy). Potentials with such behavior arise in various problems, e.g. $F(t) = t^2$ is known as the ``frame potential" (see \cite{F} and Lemma \ref{l.frame} below) and  $F(t)= |t|^p$ as $p$-frame potential  \cite{G}.  Possible optimizers for such energies naturally include orthonormal bases, other structures which exhibit similar orthogonality properties (e.g., frames), or the uniform distribution $\sigma$. 

The case of $F(t) = \arccos (t)$, rather than $F(t) = \arccos |t|$ (i.e. the sum of actual angles, rather than the acute line angles),  has been studied in \cite{C}: in this situation {\it{every}} symmetric measure is a maximizer, with a certain correction in the discrete case for odd $N$.

In the present paper, in addition to proving Theorem \ref{t.main}, we explore some other issues related to this problem. In \S \ref{s.dim},  we observe that the validity of Conjecture \ref{c2} in some dimension $d \ge 2$ implies its validity in all lower dimensions, and in \S \ref{s.d1} we revisit the one-dimensional case of Conjectures \ref{c1} and \ref{c2} (proved, e.g., in \cite{A}, \cite{B}) and provide several alternative proofs: two based on orthogonal expansions, and one based on the Stolarsky principle in discrepancy theory. We hope that these approaches may shed some light on the conjectures in the future. 

\section{The frame potential and the proof of Theorem \ref{t.main}}\label{s.proof}

Since our approach is based on bounding the potential $F(t) = \arccos |t|$ by a quadratic function, we first study the behavior of the energy with the potential $t^2$. We have the following lemma.

\begin{lemma}\label{l.frame}
For any $\mu \in \mathcal B (\mathbb S^d)$, we have 
\begin{equation}\label{e.frame}
I_{t^2} (\mu) \ge \frac{1}{d+1}.
\end{equation}
The equality above is achieved precisely for the measures $\mu$ 
whose second moment matrix is a multiple of the identity, i.e. $\displaystyle{\left[ \int_{\mathbb S^d} x_i x_j \, d\mu (x)   \right]_{i,j =1}^{d+1}  = c\,  {\bf{I}}_{d+1}}$. 
\end{lemma}

\begin{proof} We expand the square, throw away off-diagonal terms, and apply the Cauchy--Schwartz inequality:
\begin{align*}
I_{t^2} (\mu) & = \int\limits_{\mathbb S^d} \int\limits_{\mathbb S^d} ( x\cdot y )^2 \,\, d\mu (x) d\mu (y)  = \sum_{i,j = 1}^{d+1} \left( \,\,\, \int\limits_{\mathbb S^d} x_i x_j \, d\mu (x) \right)^2 \\
& \ge \sum_{i=1}^{d+1}  \left( \,\,\, \int\limits_{\mathbb S^d} x_i^2 \, d\mu (x)  \right)^2  \ge \frac{1}{d+1} \left(  \sum_{i=1}^{d+1}  \, \int\limits_{\mathbb S^d} x_i^2 \, d\mu (x)  \right)^2  = \frac1{d+1} .
\end{align*}
The characterization of minimizers follows immediately from the  inequalities above. 
\end{proof}

The energy with the potential $F(t) = t^2$ arises naturally in functional analysis and signaI processing. In \cite{F}, Benedetto and Fickus defined the  discrete ``frame potential", which up to normalization is equal to $E_{t^2} (Z)$, and proved that for $N \ge d+1$ 
\begin{equation}\label{e.BF}
E_{t^2} (Z) \ge \frac{1}{d+1}
\end{equation}
with equality achieved if and only if the set $Z = \{ z_1, \dots, z_N \}  \subset \mathbb S^d$ forms a so-called  tight frame, i.e. for each $x \in \mathbb R^{d+1}$ one has an analogue of Parseval's identity: 
$\displaystyle{c\|x \|^2 = \sum_{k=1}^N | x\cdot z_k|^2}$. It is easy to see that  Lemma \ref{l.frame} implies this discrete result (although much more is proved in \cite{F}, in particular, that every local minimizer is necessarily global). Conversely, since the minimum in \eqref{e.BF} is independent of $N$, the  weak$^*$-density of probability measures shows that the result of  Benedetto and Fickus also implies \eqref{e.frame}. Moreover, inequality \eqref{e.frame} for arbitrary  measures $\mu  \in \mathcal B (\mathbb S^d)$ has been stated in \cite{G}. We have included the nice and short proof above for the sake of completeness. \\



To prove Theorem \ref{t.main}, it suffices to demonstrate that the inequality 
\begin{equation}\label{e.mainineq}
\arccos|t| \leq \frac{\pi}{2} - \frac{69}{50} t^2
\end{equation} holds for all $t\in[-1,1]$, as then \eqref{e.frame} would imply that, for all $\mu \in \mathcal{B}( \mathbb{S}^d)$,
\begin{equation}
I( \mu) \leq \frac{\pi}{2} - \frac{69}{50}I_{t^2} (\mu) \leq \frac{\pi}{2} - \frac{69}{50(d+1)}.
\end{equation}

\noindent A simple calculus exercise shows $\arccos |t|  \leq \frac{\pi}{2} - bt^2$  for all $t\in[-1,1]$ if and only if
$$ \frac{\pi}{2} - \frac{b + \sqrt{b^2 - 1}}{2} - \arccos \left( \sqrt{ \frac{b + \sqrt{b^2 - 1}}{2b}} \right) \geq 0,$$
and a quick computation shows that for $b \leq \frac{69}{50}$  the  inequality above indeed holds, finishing the proof of   Theorem \ref{t.main}. In Figure \ref{f.1} we include the graph,  illustrating inequality \eqref{e.mainineq}. $\square$ \\

\begin{figure}
  \centering
    \begin{minipage}{0.45\textwidth}
        \centering
        \includegraphics[scale=0.6]{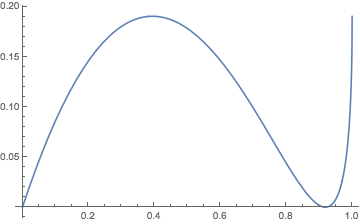}  
    \end{minipage}\hfill
    \begin{minipage}{0.45\textwidth}
        \centering
        \includegraphics[scale=0.6]{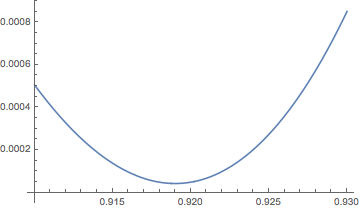} 
    \end{minipage}
    \caption{An illustration to inequality \eqref{e.mainineq}: the graph of the function 
    $\displaystyle{f (t) = \frac{\pi}2 - \frac{69}{50} \, t^2 - \arccos |t| }$ for $0\le t \le 1$.}
    \label{f.1}
\end{figure}

It is easy to see that there is very little  room for improvement via this method, and it will definitely not yield the Conjecture \ref{c2}. 

\section{Dimension reduction argument}\label{s.dim}


Let $F: [-1, 1] \rightarrow \mathbb{R}$ be a bounded, measurable function that achieves its maximum at $0$. Let $\nu \in \mathcal{B}( \mathbb{S}^k)$ and $\lambda \in \mathcal{B}( \mathbb{S}^l)$, with $k + l = d-1$. Construct a measure $\mu \in  \mathcal{B}( \mathbb{S}^d)$ as follows: 
\begin{equation}\label{e.musum}
 \mu = \alpha \widetilde{\nu} + \beta \widetilde{\lambda},
 \end{equation} where $0\le \alpha, \beta \le 1$, $\alpha + \beta =1 $, and  $\widetilde{\nu}$ and $  \widetilde{\lambda}$ are  copies of $\nu$ and $\lambda$, supported on mutually orthogonal subspheres of $\mathbb S^d$ of dimensions $k$ and $l$, respectively. It is easy to see that 
\begin{equation}\label{e.dimr}
I_F (\mu ) = \alpha^2 I_F (\nu) + \beta^2 I_F (\lambda)  + 2 \alpha \beta F(0),
\end{equation}
which easily implies that if $\mu$ is a maximizer of $I_F$ over $ \mathcal{B}( \mathbb{S}^d)$, then $\nu$ and $\lambda$ maximize $I_F$ respectively over  $\mathcal{B}( \mathbb{S}^k)$ and $ \mathcal{B}( \mathbb{S}^l)$.

The orthonormal basis measure $\nu_{ONB}$ is precisely of the form \eqref{e.musum} with $\nu$ and $\lambda$ also being orthonormal basis measures in lower dimensions. Therefore, in particular, we have just proved that the validity of Conjecture \ref{c2}  on $\mathbb S^d$ for some dimension $d\ge 2$ implies its validity on $\mathbb S^k$ for $k <d$, i.e. in all lower dimensions. 

Moreover, in our case, for $F(t) = \arccos |t|$, we have $F(0) = \frac{\pi}{2}$. Let $l=0$ and $k= d-1$. Then $I  (\lambda ) = F (1) = 0$, and thus  \eqref{e.dimr} becomes $$ I (\mu ) = \alpha^2 I  (\nu)     + \pi \alpha (1-\alpha) . $$
Optimizing this quadratic polynomial in $\alpha$ we find that for $\displaystyle{\alpha = \frac{\pi}{2 \big(\pi - I_F (\nu) \big)}}$
$$\displaystyle{I(\nu ) = \pi - \frac{\pi^2}{4 I(\mu) }}.$$
This discussion leads to the following conclusion.
\begin{proposition}[Dimension reduction] $\,\,$ 
\begin{enumerate}
\item Denote $M_d = \max \{ I (\mu):\, \mu \in \mathcal B (\mathbb S^d ) \}$. Then $$ M_{d-1} \le \pi - \frac{\pi^2}{4 M_d}.$$
\item Assume that Conjecture \ref{c2} holds on $\mathbb S^d$, i.e. $M_d = \frac{\pi d}{2(d+1)}$. Then it also holds on $\mathbb S^{d-1}$, and consequently in all lower dimensions. 
\end{enumerate}
\end{proposition}

Observe, that part (2) follows both from part (1), and, in a more general setting,  from the discussion in the beginning of this section. This  Proposition implies that, in order to prove Conjecture \ref{c2} in all dimensions $d\ge 2$, it   is enough to establish its validity for infinitely many values of $d$.

\section{The case of $\mathbb S^1$ revisited}\label{s.d1}

We now revisit the case  $d=1$, in which Conjecture \ref{c2} (and hence also Conjecture \ref{c1}) has been settled. We shall provide several new approaches to this case  (two based on orthogonal expansions, and one based on connections to discrepancy theory), which  we hope may lead to future progress in higher dimensions. 

Before we proceed, we observe  that on $\mathbb S^1$ 
\begin{equation}\label{e.all}
I ( \nu_{ONB} ) = I (\sigma) = I (\sigma_{4N} ) = \frac{\pi}{4},
\end{equation}
where $\sigma$ is the normalized uniform measure $d\sigma = \frac{d\theta}{2\pi}$, and $\displaystyle{\sigma_{4N } = \frac{1}{4N} \sum_{k=1}^{4N} \delta_{e^{\pi i k/2N}}}$ is the measure with mass equally concentrated at $4N$ equally spaced points. Hence, to prove Conjecture \ref{c2} it suffices to prove that any of these measures is a maximizer.  

\subsection{Chebyshev polynomial expansion. }\label{s.che} It is well known (see, e.g., \cite[Proposition 2.1]{D}) that  $\sigma$ is a maximizer of the energy integral  $I_F$ on $\mathbb S^1$  if and only if $F$ is a negative definite function on $\mathbb S^1$ (up to the constant term), which is equivalent to
the fact that  in the orthogonal expansion of $F$ into Chebyshev polynomials,  $\displaystyle{F(t) \sim \sum_{n=0}^\infty \widetilde{F}_n T_n  (t)}$, the coefficients of all non-constant terms are non-positive, i.e. $$ \widetilde{F}_n \le 0 \,\,\, \textup{ for } \,\,\, n\ge 1. $$
Here $T_n$ is the $n^{th}$ Chebyshev polynomial of the first kind,  that is
\begin{equation}\label{e.chepoly}
 T_n(t) = \cos( n \arccos(t)).
 \end{equation}
A straightforward, but technical   computation  (which we omit) shows that for $F(t) = \arccos|t|$,
\begin{equation}\label{e.cheb}
 \widetilde{F}_n  = \frac{1}{\pi} \int_{-1}^{1} F(t) T_n(t) (1-t^2)^{- \frac{1}{2}} dt  = \begin{cases}
\frac{\pi}{4}, & \text{ if } n = 0,\\
\frac{-4}{ \pi n^2}, & \text{ if } n \equiv 2 \text{ mod } 4, \\
0, & \text{otherwise}.
\end{cases}
\end{equation}
Therefore, $\sigma$ maximizes $I_F$ on $\mathbb S^1$, which together with \eqref{e.all} implies Conjecture \ref{c2}. 

\subsection{Fourier series}
Our second orthogonal expansion is a Fourier series. Every point in $\mathbb{S}^1$ can be defined by its angle, so with a slight abuse of notation we can rewrite our energy integral as
$$ I_F(\mu) = \int_{ \mathbb{T}} \int_{\mathbb{T}} F( \cos( \theta - \phi)) d \mu(\theta) d\mu( \phi) = \int_{ \mathbb{T}} \int_{\mathbb{T}} G (  \theta - \phi) d \mu(\theta) d\mu( \phi) .$$
Let  $ G( \theta) = F( \cos( \theta)) $ 
 be an even function.  Then for $\nu \in \mathcal B (\mathbb S^d)$  defined by $d \nu(\theta) = \frac{d\mu(\theta) + d\mu( -\theta)}{2}$ for all $\theta \in \mathbb{T}$, we have $I_F(\mu) = I_F( \nu)$. Thus, for the rest of this section, we may assume that $\mu$ is an even measure. As long as  the Fourier series of $G$, which is a cosine series $ \displaystyle{G( \theta ) \sim \sum_{n=0}^\infty \widehat{G} (n) \cos n\theta }$,  converges absolutely,  we can use it to compute the energy:
\begin{equation}\label{e.enfourier}
\begin{aligned}
I_F(\mu) & = \sum_{n=0}^{ \infty} \widehat{G}(n) \int_{ \mathbb{T}} \int_{ \mathbb{T}}  \cos( n( \theta - \phi)) \;  d \mu(\theta) d \mu( \phi) \\
& = \sum_{n=0}^{ \infty} \widehat{G}(n) \int_{ \mathbb{T}} \int_{ \mathbb{T}}  \bigg( \cos( n \theta) \cos(n \phi) + \sin( n \theta) \sin(n \phi)  \bigg) \;  d \mu(\theta) d \mu( \phi). \\
\end{aligned}
\end{equation}
Since $\mu$ is an even measure, the sines do not contribute to the integral, and we have
\begin{equation}\label{e.enfourier1}
I_F(\mu) 
 = \sum_{n=0}^{ \infty} \widehat{G}(n) \left( \, \int_{ \mathbb{T}}  \cos( n \theta)   d \mu(\theta) \right)^2 \le  \widehat{G}(0) + \sum_{n\ge 1 : \,\, \widehat{G}(n) \ge 0 } \widehat{G}(n) 
\end{equation}
The equality above is achieved if and only if $\displaystyle{ \int_{ \mathbb{T}}  \cos( n \theta)   d \mu(\theta) = 0} $ for each value of $n\ge 1$ for which $\widehat{G}(n) < 0$ {and} $\displaystyle{ \int_{ \mathbb{T}}  \cos( n \theta)   d \mu(\theta) = 1} $ for each value of $n\ge 1$ such that $\widehat{G}(n) >0$.  

\begin{lemma}\label{l.points} Let $G$ be even and have absolutely convergent Fourier (cosine) series and let $\sigma_N$ be the probability measure generated by point masses at $N$ equally spaced points $$ \sigma_N = \frac1{N} \sum_{k=0}^{N-1} \delta_{2\pi k/N} . $$ 
Assume that for all $n\ge 1$ we have $\widehat{G} (n) \ge 0$  if  $n$ is a multiple of $N$, and $\widehat{G} (n) \le 0$ otherwise.  Then the measure $\sigma_N$ maximizes $I_F$ over all probability measures on $\mathbb T$.
\end{lemma}
The proof easily follows from the discussion above and the fact that $$ \int_{ \mathbb{T}}  \cos( n \theta)   d \sigma_N (\theta) = \frac1{N}  \sum_{k=0}^{N-1} \cos \frac{2\pi n k}{N} =   \begin{cases}
1, & \text{ if  $n$ is a multiple of $N$},\\
0, & \text{ otherwise}. \end{cases}$$

\noindent {\it{Remark:}}  the conditions on $\widehat{G} (n)$ in Lemma \ref{l.points} become also necessary,  if we assume that $\sigma_N$ maximizes $I_F$ over all {\it{signed}} (not just non-negative) measures of total mass one. This can be proved by considering signed measures of the form $d\mu= d\sigma - \cos n\theta \,d\sigma$.

\vskip2mm

Returning to  our specific case, $ G( \theta) = F( \cos( \theta)) = \min\{ |\theta|, \pi- |\theta| \}$, we observe that   according to  definition \eqref{e.chepoly} and relation  \eqref{e.cheb}, we have $ \widehat{G}  (n) =  \widetilde{F}_n = - \frac{4}{\pi n^2}$  whenever $ n \equiv 2 \text{ mod } 4$, and $\widehat{G}  (n) = 0$ for all other values of $n\ge 1$.  Therefore, Lemma \ref{l.points} with $N=4$ applies and $\sigma_4  = \frac{1}{4} \big(  \delta_{0} + \delta_{\frac{\pi}{2}} + \delta_{ \pi} + \delta_{  \frac{3 \pi}{2}} \big)$ maximizes $I_F$. Symmetry implies that $\nu_{ONB} =  \frac{1}{2} \big( \delta_0 + \delta_{\frac{\pi}{2}} \big)$ is also a maximizer, i.e. Conjecture \ref{c2} holds on $\mathbb  S^1$. \\

Notice  that alternatively, since $\widehat{G} (n) \le 0$ for all $n\ge 1$,  one could deduce from \eqref{e.enfourier1} that $I_F (\mu ) \le \widehat{G} (0) = I_F (\sigma)$. Therefore $\sigma$ is a maximizer of the energy integral $I_F$, and, due to \eqref{e.all}, so is $\nu_{ONB}$, leading to another proof of Conjecture \ref{c2} (almost identical to \S \ref{s.che}). We, however, wanted to introduce Lemma \ref{l.points} into play. 


\subsection{Stolarsky principle}
Finally, our last approach is based on an intriguing connection between energy optimization and discrepancy.  The classical {\it{Stolarsky invariance principle}}  \cite{stol}  relates the $L^2$  spherical cap discrepancy to a certain discrete energy (sum of pairwise Euclidean distances, i.e. $F(x\cdot y ) = \| x-y\|$).  Some generalizations of this principle were obtained in \cite{D,C}. Moreover, in \cite{C} a version of this principle has been used to characterize the extremizers of the energy  integral and discrete energy with the potential $F(t) = \arccos t$ (i.e. $F(x\cdot y)$ represents the angle or the geodesic distance between $x$ and $y\in \mathbb S^d$). We shall apply this idea to our  problem -- the main task is to find an appropriate version of  $L^2$ discrepancy.


 For   $x \in \mathbb{S}^1$, define the {\it{antipodal quadrants}} in the direction of $x$ as  $Q(x) = \{ y : |x \cdot y| > \frac{\sqrt{2}}{2} \}$, i.e. $Q(x)$ is a union of two symmetric quartercircle arcs  with midpoints  at $x$ and $-x$.  Observe that the size of the  intersection of two such sets depends linearly on the acute angle between their directions: 
 \begin{equation}\label{e.intersect}
 \sigma \big(  Q(y) \cap Q(z) \big) = \frac12 - \frac1{\pi} \arccos |y\cdot z|. 
 \end{equation}
We define  the $L^2$ discrepancy of a point set $Z = \{z_1, ..., z_N \} \subset \mathbb S^1$ with respect to these antipodal quadrants as the quadratic average of the difference between the empirical measure  $\frac{\# \{ Q(x) \cap Z \}}{N}$ and the uniform measure $\sigma \big(Q(x) \big) = \frac12$. 
  \begin{equation}\label{e.disc}
D^2_{L^2, \textup{quad}}(Z) = \int_{ \mathbb{S}^1} \left| \frac{1}{N} \sum_{i=1}^{N} \mathbbm{1}_{Q(x)}(z_i) - \sigma \big(Q(x)\big) \right|^2 d\sigma(x).
 \end{equation}
 More generally, for a measure $\mu \in \mathcal B (\mathbb S^1)$ we can define its discrepancy as 
 \begin{equation}\label{e.discmu}
 D^2_{L^2, \textup{quad}}(\mu) =  \int_{ \mathbb{S}^1} \left| \mu \big( Q(x) \big)  - \sigma \big(Q(x)\big) \right|^2 d\sigma(x),
 \end{equation}
 i.e.  $D^2_{L^2, \textup{quad}}(Z) = D^2_{L^2, \textup{quad}} \big( \frac1{N} \sum \delta_{z_i} \big)$.  
 The following version of the Stolarsky principle holds:
 
 \begin{lemma}[Stolarsky principle for quadrants]\label{l.stol} For any measure $\mu \in \mathcal B (\mathbb S^1)$ we have 
 \begin{equation}\label{e.stol1}
  D^2_{L^2, \textup{quad}}(\mu) = \frac{1}{\pi} \big(   I (\sigma)  - I (\mu) \big) = \frac14 - \frac1{\pi} I(\mu) . 
 \end{equation} 
In particular, for a discrete  point set $Z = \{z_1, ..., z_N \} \subset \mathbb S^1$ 
 \begin{equation}\label{e.stol2}
  D^2_{L^2, \textup{quad}}(Z) =  \frac14 - \frac1{\pi} E (Z) . 
 \end{equation} 
 \end{lemma}
 
 \begin{proof} We use relations \eqref{e.discmu} and \eqref{e.intersect} to obtain
 \begin{equation*}
\begin{aligned}
D^2_{L^2, \textup{quad}} (\mu)  & =   \int\limits_{ \mathbb{S}^1} \int\limits_{ \mathbb{S}^1}\int\limits_{ \mathbb{S}^1}  \mathbbm{1}_{Q(x)} (y) \cdot \mathbbm{1}_{Q(x)}(z) \, d\sigma(x) \, d\mu (y) \, d \mu(z) - \int\limits_{ \mathbb{S}^1} \int\limits_{ \mathbb{S}^1}  \mathbbm{1}_{Q(x)}(y) d \sigma(x) \, d\mu (y)  + \frac{1}{4} \\
& =  \int\limits_{ \mathbb{S}^1} \int\limits_{ \mathbb{S}^1}   \sigma\big( Q(y) \cap Q(z)\big) - \frac{1}{4} = \frac14 - \frac1{\pi} \int\limits_{ \mathbb{S}^1} \int\limits_{ \mathbb{S}^1}  \arccos |y\cdot z| \, d\mu (y) \, d \mu(z) ,
\end{aligned}
\end{equation*}
which proves \eqref{e.stol1}. 
 \end{proof}
 
\noindent  {\it{Remark:}}  discrepancy with respect to similar sets, but with arbitrary apertures, (``wedges") on $\mathbb S^d$ arose in \cite{BL}  in connection to sphere tessellations and one-bit sensing, which led to a Stolarsky principle involving energies with the potential $F(t) = (\arcsin t)^2$. \\

 Our Stolarsky principle \eqref{e.stol1}  proves  that $I(\mu) \le \frac{\pi}4$  and   $\sigma$ maximizes this energy, and by \eqref{e.all} so does $\nu_{ONB}$ proving Conjecture \ref{c2}.  
In the discrete case, \eqref{e.stol2} shows that $E(Z) \le \frac{\pi}4$ for even $N$ and $E(Z) \le \frac{\pi}{4}\cdot \frac{N^2-1}{N^2}$ for odd $N$. It also allows one to characterize the maximizers of $E(Z)$ and prove Conjecture \ref{c1}: \eqref{e.stol2} implies that   maximizing $E(Z)$ is equivalent to minimizing the discrepancy $D^2_{L^2, \textup{quad}}(Z)$. It is easy to see from \eqref{e.disc} that this happens exactly when the following holds:  for $\sigma$-almost every $x\in \mathbb S^d$ the difference between the number of points of $Z$ in $Q(x)$ and in its complement $\big( Q(x) \big)^c$ is zero (when $N$ is even) or $\pm 1$ (when $N$ is odd). Equivalently, this should hold for any $x$ such that the boundary of $Q(x)$ doesn't intersect $Z$, which recovers the characterization in \cite{B}. The  extremizing configurations in Conjecture \ref{c1} obviously satisfy this condition. \\

Unfortunately, none of these approaches immediately extends to higher dimensions: for $d\ge 2$ the function $F(t) = \arccos |t|$ is not negative definite, \eqref{e.all} fails, and absence of analogues of \eqref{e.intersect} prevents one from obtaining a Stolarsky principle. However, we hope that some of these ideas with additional ingredients may bring progress on the conjectures in the future.\\

\noindent \textbf{Acknowledgment.}   This  work is supported by the NSF grant DMS 1665007 (Bilyk) and the NSF Graduate Research  Fellowship 00039202 (Matzke).




\bibliographystyle{amsplain}

\end{document}